\documentclass[11pt,reqno]{amsart}
\usepackage{todonotes}
\usepackage{yfonts}
\usepackage{hyperref}
\hypersetup{
    colorlinks = true,
    citecolor=blue,
    filecolor=black,
    linkcolor=red,
    urlcolor=black
}
\usepackage{pgfplots}
\usepgfplotslibrary{groupplots}
\pgfplotsset{compat=1.6}
\pgfplotsset{every axis title/.append style={at={(0.6,1.1)}}}

\usepackage{pgf, tikz}
\usetikzlibrary{automata, positioning, arrows, decorations.pathreplacing}
\usepackage{bbm}

\usepackage[utf8]{inputenc}

\usepackage{graphicx}
\usepackage[]{algorithm2e}

\usepackage{hyperref}
\hypersetup{
    colorlinks = true,
    citecolor=blue,
    filecolor=black,
    linkcolor=red,
    urlcolor=black
}

\usepackage{amssymb}
\usepackage{latexsym}
\usepackage{amsbsy}
\usepackage{amsfonts}
\usepackage{appendix}
\usepackage{tikz}
\usetikzlibrary{decorations.pathmorphing}
\usetikzlibrary{decorations.pathreplacing}
\usepackage{pstricks}
\usepackage{graphicx}
\usepackage[utf8]{inputenc}
\usepackage{enumitem}

\numberwithin{equation}{section}
\newtheorem{theorem}{Theorem}[section]

\newtheorem{proposition}[theorem]{Proposition}
\newtheorem{lemma}[theorem]{Lemma}

\newtheorem{definition}[theorem]{Definition}
\newtheorem{assumption}[theorem]{Assumption}

\theoremstyle{definition}
\newtheorem{auxremark}[theorem]{Remark}

\newdimen\AAdi%
\newbox\AAbo%
\def\AAk#1#2{\setbox\AAbo=\hbox{#2}\AAdi=\wd\AAbo\kern#1\AAdi{}}%

\makeatletter
\def\eqlabel#1{\def\@currentlabel{#1}}

\def\formula#1{\def\@tempa{#1}\let\@tempb\theequation\def\theequation{%
\hbox{#1}}\def\@currentlabel{(\theequation)}$$}
\def\endformula{\leqno\hbox{(\@tempa)}$$\@ignoretrue\let\theequation\@tempb}

\def\given{\hskip5\p@\relax\vrule\@width.4\p@\hskip5\p@\relax}

\newcommand{\open}[1]{%
\par\normalfont\topsep6\p@\@plus6\p@\trivlist\item[\hskip\labelsep\itshape#1%
\@addpunct{.}]\ignorespaces}

\DeclareRobustCommand{\close}[1]{%
  \ifmmode 
  \else \leavevmode\unskip\penalty9999 \hbox{}\nobreak\hfill
  \fi
  \quad\hbox{$#1$}}
\makeatother

\newlength{\toskip}

\begin{document}
\title[Minimal Positive Solution]{Transcendental solution to linear coefficient non-homogeneous second order recurrence relation with constant non-homogenity.}

\author[J. W. Fischer]{\textbf{\quad {Jens Walter} Fischer$^{\spadesuit\clubsuit}$ \, \, }}
\address{{\bf {Jens Walter} FISCHER},\\ Institut de Math\'ematiques de Toulouse. CNRS UMR 5219. \\
Universit\'e Paul Sabatier
\\ 118 route
de Narbonne, F-31062 Toulouse cedex 09.} \email{jens.fischer@math.univ-toulouse.fr}

\begin{abstract}
	Second order recurrence relations of real numbers arise form various applications in discrete time dynamical systems as well as in the context on Markov chains. Solutions to the recurrence relations are fully defined by the first two initial values as well as the recurrence formula. We calculate in this work explicitly as a function of $a_1$ the minimal positive solution $(a_i)_{i\in\mathbb{N}}$ to non-homogeneous second order recurrence relation with affine coefficients when the non-homogeneity is constant and negative, and the first initial value equals $a_0=0$. We show that rational coefficients lead to a sequence of transcendental numbers.\\
	Additionally, we prove that this sequence is the only bounded solution when varying $a_1$, converges to $0$ and obtain the convergence speed in $O(i^{-1})$. We comment in the last section further on the choice of rational parameters in the recurrence relation and we make a link to the impossibility of obtaining computer based visualizations of the minimal positive solution $(a_i)_{i\in\mathbb{N}}$. 
\end{abstract}
\bigskip

\maketitle

\textsc{$^{\spadesuit}$  Universit\'e de Toulouse}
\smallskip

\textsc{$^{\clubsuit}$ University of Potsdam}
\smallskip

\bigskip

\bigskip

\section{Introduction}\label{sec:second_order_rec_relations}
	Non-homogeneous second order recurrence relations with constant non-homogenity of sequences of real numbers $(a_i)_{i\in\mathbb{N}}$ have the general form
	\begin{equation}\label{eq:second_order_rec_relations}
		\begin{cases}
			(a_0,\,a_1) = (q,a)\\
			\tilde{\lambda}_i a_{i+1}= \alpha + \tilde{\omega}_i a_i - \tilde{\mu}_i a_{i-1},& i\geq 1
		\end{cases}
	\end{equation}
	where $q,a,\alpha\in\mathbb{R}$ and $(\tilde{\lambda}_i)_{i\in\mathbb{N}}, (\tilde{\mu}_i)_{i\in\mathbb{N}}, (\tilde{\omega}_i)_{i\in\mathbb{N}}$ sequences of real numbers. Minimal positive solutions to second order recurrence relations come into play, when considering a family of sequences $\mathcal{A}=\left\{\left(a_i^{(a)}\right)_{i\in\mathbb{N}}|a\in\mathbb{R}^+\right\}$ defined by the same recurrence relation \eqref{eq:second_order_rec_relations} but with only the first initial value $a_0^{(a)}\geq 0$ being fixed and the second one $a_1^{(a)}=a$ being variable but positive. A minimal positive solution $(\hat{a}_i)_{i\in\mathbb{N}}$ of  the recurrence relation \eqref{eq:second_order_rec_relations} then satisfies $\hat{a}_i>0$ for all $i\geq 1$ and for any other sequence $(a_i)_{i\in\mathbb{N}}\in \mathcal{A}$ with $a_i>0$ for all $i\geq 1$ also $a_i\geq \hat{a}_i$ is satisfied for all $i\geq 1$. In particular, we find a link to transcendental numbers.
	\begin{definition}[Algebraic and Transcendent Numbers, \cite{Ono1990}]
		A number $z \in\mathbb{C}$ is called algebraic, if it is a non-zero root of some rational polynomial $p(z)=\sum_{n=0}^N q_n z^n$, i.e., $q_n\in\mathbb{Q}$ for all $n=0,\hdots,N$. Otherwise, the number $z$ is called a transcendental number.
	\end{definition}
	In what follows, we specify the recurrence relation of interest for us in this work. We investigate the behavior of the sequence $(a_i)_{i\in\mathbb{N}}$ defined by the recurrence relation \eqref{eq:second_order_rec_relations} assuming $q=0$, $a>0$, $\alpha \leq 0$ and write $\alpha=-\xi$ for some $\xi\geq 0$ and $\tilde{\omega}_i>\tilde{\lambda}_i + \tilde{\mu}_i$. Due to this assumption we can rewrite $ \tilde{\omega}_i=\tilde{\gamma}_i+\tilde{\lambda}_i + \tilde{\mu}_i$ for some positive $\tilde{\gamma}_i$ and for all $i\in\mathbb{N}$. From this we obtain the following form of the recurrence relation \eqref{eq:second_order_rec_relations}
	\begin{equation}\label{eq:main_recurrence_relation}	
		\begin{cases}
			(a_0,\,a_1) = (0,a)\\
			\tilde{\lambda}_i a_{i+1}= - \xi + ( \tilde{\lambda}_i + \tilde{\mu}_i + \tilde{\gamma}_i) a_i - \tilde{\mu}_i a_{i-1},& i\geq 1.
		\end{cases}
	\end{equation}
	For all quantitative results in this work, in particular, the rates of convergences for the the sequences resulting from \eqref{eq:main_recurrence_relation} we assume the following form of the concerned sequences $(\tilde{\lambda}_i)_{i\in\mathbb{N}}, (\tilde{\mu}_i)_{i\in\mathbb{N}}, (\tilde{\omega}_i)_{i\in\mathbb{N}}$.
	\begin{assumption}\label{amp:affine_rates}
		Let 
		\begin{enumerate}
			\item $\mathfrak{n}\in\mathbb{N}^{\ast}$, $\gamma,\lambda,\mu> 0$, $\xi\geq 0$,
			\item $\tilde{\gamma}_i=\gamma\,(i+\mathfrak{n}),\tilde{\lambda}_i= \lambda\,(i+\mathfrak{n})$ and $\tilde{\mu}_i=\mu\,(i+\mathfrak{n})$ for $i\in\mathbb{N}$.
		\end{enumerate}
	\end{assumption}
	We discussed generalizations under reduced assumptions on the coefficients in the Outlook.
	\section{Results on the minimal positive solution}
	We make a distinction between two cases, which behave qualitatively completely differently, the case  $\xi=0$ and the case $\xi>0$. The following Theorem considers the case $\xi=0$ and describes fully the sequence defined in \eqref{eq:main_recurrence_relation}. The case $\xi=0$ corresponds effectively the homogeneous case and is, therefore, easily treatable and can be covered with straight forward arguments. We need them as preliminary results in the following section. The case $\xi>0$, which we consider in this work, is more intricate but we obtain, nonetheless, the explicit form of all solutions. In particular, we obtain the following result on transcendental solutions which correspond to minimal positive solutions of Equation \eqref{eq:main_recurrence_relation}. 
	\begin{theorem}\label{thm:exclusion_rational_of_both}
		Denote by $\hat{a}\in\mathbb{R}^+$ the initial value $a_1$, such that the resulting solution is the positive minimal solution to \eqref{eq:main_recurrence_relation}. 
		Suppose $\gamma,\lambda,\mu>0$ and let $\underset{\bar{}}{a}<\bar{a}$ be the distinct real zeros of the polynomial 
		\begin{equation*}
			X\mapsto\mu X^2-(\lambda+\mu+\gamma)X+\lambda.
		\end{equation*} 
		Under assumption \ref{amp:affine_rates}, either $\gamma,\lambda,\mu\in\mathbb{Q}$ or $\hat{a}\in\mathbb{Q}$ but not both.
	\end{theorem}
	Theorem \ref{thm:exclusion_rational_of_both} is based on our complete analysis of the solutions to \eqref{eq:main_recurrence_relation}, in which we find the necessary and sufficient condition for the minimal positive solution as well as a complete description of the limit behavior of the minimal positive solution as $i\to\infty$ as well as for all other solutions. 
	\begin{theorem}\label{thm:explicit_form_any_solution}
		Denote by $\underset{\bar{}}{a},\bar{a}\in(0,1)$ with $\underset{\bar{}}{a}<\bar{a}$ the real zeros of the polynomial 
		\begin{equation*}
			X \mapsto \mu X^2-(\lambda+\mu+\gamma)X+\lambda.
		\end{equation*}  
		Under Assumption \ref{amp:affine_rates} any solution $(a_i)_{i\in\mathbb{N}}$ to equation \eqref{eq:main_recurrence_relation} with initial value  $(0,x)$ where $x>0$ has the form
		\begin{equation}
		a_i = \lambda\, x\, c \left(\dfrac{1}{\underset{\bar{}}{a}^i}-\dfrac{1}{\bar{a}^i}\right)-\sum_{k=1}^{i-1}\dfrac{ c\xi }{k+ \mathfrak{n} } \left(-\dfrac{1}{\bar{a}^{i-k}}+ \dfrac{1}{\underset{\bar{}}{a}^{i-k}}\right),
		\end{equation}
		with $ c  = (\sqrt{(\lambda+\mu+\gamma)^2-4\lambda\mu})^{-1}$.
	\end{theorem} 
	Since a sequence $(a_i)_{i\in\mathbb{N}}$ defined by a recurrence relation of order $2$ is fully defined by its initial values $a_0$ and $a_1$ and having fixed $a_0=0$, the minimality of the solution to \eqref{eq:main_recurrence_relation} is solely dependents on $a_1$. The following Theorem \ref{thm:minimal_solution_result} gives explicitly the form of said initial value and shows the behavior of $a_i$ as $i\to\infty$.
	\begin{theorem}\label{thm:minimal_solution_result}
		Under Assumption \ref{amp:affine_rates} and defining
		\begin{equation}\label{eq:x_hat_minimal_solution}
			\hat{a}	:= \dfrac{\xi}{\lambda}\left(\sum_{k=0}^{\infty}\dfrac{\underset{\bar{}}{a}^k}{\mathfrak{n}+k} -  \mathfrak{n} ^{-1}\right)
		\end{equation}	 	
		the sequence $(a_i)_{i\in\mathbb{N}}$ satisfying \eqref{eq:main_recurrence_relation} with initial value $(a_0, a_1) = (0,\hat{a})$ is the minimal positive solution to \eqref{eq:main_recurrence_relation} and $a_i = \mathcal{O}(i^{-1})$ as $i\to\infty$.
	\end{theorem} 
	Indeed, the proof of this result leads to the conclusion that the only initial value $a_1$ for which the solution to \eqref{eq:main_recurrence_relation} is bounded. The value $\hat{a}$ separates the real line into initial values $a_1'<\hat{a}$ for which $a_i'\to-\infty$ as $i\to\infty$ and values $a_1'>\hat{a}$ for which $a_i'\to\infty$ as $i\to\infty$.\par
	We discuss, in the following section, the proofs of the Theorems \ref{prop:monotony_thm_by_recurrence_xi_zero} to \ref{thm:minimal_solution_result} as well as additional properties and supplementary results.
\section{Proof of Theorems}
	In this section we prove the main theorems presented in Section \ref{sec:second_order_rec_relations} and present as well as prove a couple of preliminary results which are, in particular, the preparation for the proofs of Theorems \ref{thm:explicit_form_any_solution} and \ref{thm:minimal_solution_result}.
	\begin{proposition}\label{prop:monotony_thm_by_recurrence_xi_zero}
		Under Assumption \ref{amp:affine_rates} with $\xi=0$ consider for $a>0$ the sequence $(a_i)_{i\in\mathbb{N}}$ defined by the recurrence relation \eqref{eq:main_recurrence_relation} with $\xi=0$.
		\begin{equation}\label{eq:recurrence_xi_zero}	
			\begin{cases}
				(a_0,\,a_1) = (0,a)\\
				\tilde{\lambda}_i a_{i+1} = ( \tilde{\lambda}_i + \tilde{\mu}_i + \tilde{\gamma}_i)a_i - \tilde{\mu}_i a_{i-1},& i\geq 1.
			\end{cases}
		\end{equation}
		Then $(a_i)_{i\in\mathbb{N}}$ is non-decreasing and $\underset{i\to\infty}{\lim} a_i=\infty$.
	\end{proposition}
	Indeed, we can quantify the speed of divergence for the sequence $(a_i)_{i\in\mathbb{N}}$ defined in Proposition \ref{prop:monotony_thm_by_recurrence_xi_zero} by the following Proposition \ref{prop:convergence_speed_xi_zero}.
	\begin{proposition}\label{prop:convergence_speed_xi_zero}
		Under Assumption \ref{amp:affine_rates} with $\xi=0$ consider for $a>0$ the solution $(a_i)_{i\in\mathbb{N}}$ to the recurrence relation \eqref{eq:recurrence_xi_zero} with initial condition $(0,a)$. Then the sequence $(a_i a_{i+1}^{-1})_{i\in\mathbb{N}}$ converges to the smaller real zero $\underset{\bar{}}{a}\in(0,1)$ of the polynomial 
		\begin{equation*}
			X \mapsto \mu X^2-(\lambda+\mu+\gamma)X+\lambda.
		\end{equation*}  
	\end{proposition} 
	We start with the proof of Proposition \ref{prop:monotony_thm_by_recurrence_xi_zero}. 
	\begin{proof}[Proof of Proposition \ref{prop:monotony_thm_by_recurrence_xi_zero}]
		Note that \eqref{eq:recurrence_xi_zero} is equivalent to
		\begin{equation*}
			\tilde{\lambda}_i (a_{i+1}-a_i) =  \tilde{\mu}_i (a_i - a_{i-1}) +\tilde{\gamma}_i (a_i-a_0).
		\end{equation*}
		First, $a_1-a_0 = a > 0$ by assumption. Fix $i\geq 1$ and suppose that for $j\leq i-1$ the inequality $a_{j+1}-a_j\geq 0$ holds. Hence, in particular, $a_i\geq 0$. Moreover
		\begin{eqnarray*}
		    \tilde{\lambda}_i( a_{i+1} - a_{i} ) =  \tilde{\mu}_i (a_i - a_{i-1}) +\tilde{\gamma}_i \sum_{j=1}^i(a_j-a_{j-1})\geq 0.
		\end{eqnarray*}
		By induction $(a_i)_{i\in\mathbb{N}}$ is non-decreasing and, thus, $a_i\geq 0$ for all $i\in\mathbb{N}$. Moreover, knowing now that for all $i\in\mathbb{N}$ it holds $a_i \geq a_{i-1}\geq 0$ we obtain $\tilde{\lambda}_i a_{i+1} \geq (\tilde{\lambda}_i + \tilde{\gamma}_i)a_i$. Hence, for all $i\in\mathbb{N}$ it holds
		\begin{eqnarray*}
		    a_{i+1} &\geq & \dfrac{\tilde{\lambda}_i + \tilde{\gamma}_i}{\tilde{\lambda}_i}a_i = \left(1+\dfrac{\gamma}{\lambda}\right) a_i\geq \hdots \geq \left(1+\dfrac{\gamma}{\lambda}\right)^i a
		\end{eqnarray*}
		and, thus, $\underset{i\to\infty}{\lim}a_i = \infty$ with at least a geometric rate.
	\end{proof}
	Indeed, the proof can be extend to any case where $\inf_{i\in\mathbb{N}}\tilde{\gamma}_i\,\tilde{\lambda}_i^{-1}>C$ for some positive constant $C>0$. The assumption on the explicit form of the rates can, therefore, be loosened.\par 
	Next, we consider the proof of Proposition \ref{prop:convergence_speed_xi_zero}, exploiting the form of the coefficients given by Assumption \ref{amp:affine_rates}.
	\begin{proof}[Proof of Proposition \ref{prop:convergence_speed_xi_zero}]
		It holds for $i\geq 2$ under Assumption \ref{amp:affine_rates}
		\begin{eqnarray*}
			a_{i-1} a_{i}^{-1} &=& \dfrac{\tilde{\lambda}_{i-1}a_{i-1}}{(\tilde{\lambda}_{i-1} + \tilde{\gamma}_{i-1} + \tilde{\mu}_{i-1})a_{i-1} - \tilde{\mu}_{i-1}a_{i-2}}\\
			&=& \dfrac{\lambda a_{i-1}}{(\lambda + \gamma + \mu )a_{i-1} - \mu\, a_{i-2}}.
		\end{eqnarray*}
		Set $z_i = a_i a_{i+1}^{-1}$ for $i\geq 1$ we therefore have
		\begin{eqnarray*}
			z_1 &=& \dfrac{\lambda }{\lambda + \gamma + \mu}, \qquad
			z_i = \dfrac{\lambda}{\lambda + \gamma + \mu - \mu z_{i-1}}.
		\end{eqnarray*}
		Consider the map $\phi:[0,1]\to[0,1]$ defined by $\phi(x)= \dfrac{\lambda}{\lambda + \gamma + \mu - \mu x}$. Since $0<\phi(0)<\phi(1) < 1$ and $\phi$ is strictly increasing $\phi$ has a unique fixed point $\underset{\bar{}}{a}\in(0,1)$ such that $z_i\to \underset{\bar{}}{a}$ as $i\to\infty$ and we obtain $\lambda = (\lambda + \mu + \gamma)\underset{\bar{}}{a} - \mu \underset{\bar{}}{a}^2$. The value $\underset{\bar{}}{a}< 1$ is thus given by 
		\begin{equation*}
			\underset{\bar{}}{a} = \dfrac{\lambda + \mu+\gamma - \sqrt{(\lambda + \mu+\gamma)^2 - 4\lambda\mu}}{2\mu}.
		\end{equation*}
	\end{proof}
	To prove Theorem \ref{thm:explicit_form_any_solution} and Theorem \ref{thm:minimal_solution_result} we first need to consider a couple of preliminary results. We focus firstly on the dependence of the solution of \eqref{eq:main_recurrence_relation} with respect to the initial value $(0,a)$ for $a \in \mathbb{R}$.  
	\begin{proposition}\label{lem:at_most_one_bdd_sol}
		Under Assumption \ref{amp:affine_rates} let $(a_i)_{i\in\mathbb{N}}$ be a solution to \eqref{eq:main_recurrence_relation}. Then there is at most one initial value $(0,a)$ with $a\in\mathbb{R}$ such that $(a_i)_{i\in\mathbb{N}}$ is bounded.
	\end{proposition}
	\begin{proof}
		Consider two solutions $(a_i)_{i\in\mathbb{N}}$ and $(a_i')_{i\in\mathbb{N}}$ of the recurrence relation \eqref{eq:main_recurrence_relation} for initial conditions $(0,a)$ and $(0,x')$, respectively, with $x>x'$. The sequence $(\Delta_i)_{i\in\mathbb{N}}:=(a_i - a_i')_{i\in\mathbb{N}}$ satisfies relation \eqref{eq:recurrence_xi_zero} with initial condition $(0,x-x')$. By Proposition \ref{prop:monotony_thm_by_recurrence_xi_zero}, $(\Delta_i)_{i\in\mathbb{N}}$ is a non-decreasing sequence tending to $\infty$ as $i\to\infty$. If there were two bounded sequences for different initial values $x$ and $x'$ also their difference would be bounded which is a contradiction. 
	\end{proof}
	The following lemmas yield step by step an initial value $(0,\hat{a})$ for which the solution $(a_i)_{i\in\mathbb{N}}$ is in fact bounded. 
	\begin{lemma}\label{prop:behavior_above_below_threshhold}
		Under Assumption \ref{amp:affine_rates} consider the solution $(a_i)_{i\in\mathbb{N}}$ to the recurrence relation \eqref{eq:main_recurrence_relation}. There exists a value $\hat{a}>0$ such that,
		\begin{enumerate}
			\item if $a_1 < \hat{a}$, $\lim_{i\to\infty}a_i = -\infty$,
			\item if $a_1 > \hat{a}$, $\lim_{i\to\infty}a_i = +\infty$ .
		\end{enumerate}
	\end{lemma}
	\begin{proof}
		Consider all unbounded solutions $(a_i)_{i\in\mathbb{N}}$  to \eqref{eq:main_recurrence_relation}. Recall that for solutions $(a_i)_{i\in\mathbb{N}}$ and $(a_i')_{i\in\mathbb{N}}$ of \eqref{eq:main_recurrence_relation} with initial values $x,x'\in \mathbb{R}$ satisfying $x > x'$ we define the sequence $(\Delta_i)_{i\in\mathbb{N}}=(a_i - a_i')_{i\in\mathbb{N}}$. Since $(\Delta_i)_{i\in\mathbb{N}}$ is by Proposition \ref{prop:convergence_speed_xi_zero} non-decreasing we obtain that if for two initial values $x,x'\in \mathbb{R}$ the inequality $a_1 = x > x' = a_1'$ holds then $a_i > a_i'$ for all $i\geq 1$.\par 
		Consider the case where there is a $i_0\in\mathbb{N}$ such that $a_{i_0}<0$ and $a_{i_0-1} \geq 0$. Note that if there is a $k\in\mathbb{N}$ such that $a_{k}<0$ such a pair $(a_{i_0},a_{i_0-1})$ may always be found since $a_0 = 0$. Then, because $\tilde{\gamma}_i > 0$ for all $i\in\mathbb{N}$, we have
		\begin{equation}\label{eq:negative_recurrence_relation}
			\tilde{\lambda}_{i_0} (a_{i_0 + 1} - a_{i_0}) = -\xi + \tilde{\gamma}_{i_0} a_{i_0} + \tilde{\mu}_{i_0}(a_{i_0} - a_{i_0-1}) < 0. 
		\end{equation}
		Thus, $a_{i_0+1} < a_{i_0} < 0$ and, because $\tilde{\gamma}_{i_0} < \tilde{\gamma}_{i_0+1}$, we have 
		\begin{equation*}
			0 > \tilde{\gamma}_{i_0} a_{i_0} > \tilde{\gamma}_{i_0} a_{i_0+1} > \tilde{\gamma}_{i_0+1}a_{i_0+1}.
		\end{equation*}
		Hence, inductively we see that starting from $i_0$ the sequence $(a_i)_{i\geq {i_0}}$ is decreasing and negative. In particular, multiplying the recurrence relation \eqref{eq:negative_recurrence_relation} with $-1$ we obtain that for all $i\geq i_0$
		\begin{equation}
			-a_{i + 1} > \dfrac{\xi}{\tilde{\lambda}_i} + \dfrac{\tilde{\lambda}_i + \tilde{\gamma}_i}{\tilde{\lambda}_i}(-a_i)=\dfrac{\xi}{\tilde{\lambda}_i}+\left(1+\dfrac{\gamma}{\lambda}\right)(-a_i) 
		\end{equation}	
		 such that $(a_i)_{i\geq {i_0}}$ diverges to $-\infty$ as $i\to\infty$. In particular, once the sequence $(a_i)_{i\in\mathbb{N}}$ becomes negative, it stays negative.\par
		Secondly, consider a positive solution $(a_i')_{i\geq 0}$ of \eqref{eq:main_recurrence_relation}. Then, by the unboundedness assumption, it holds that for all $C>0$ there is a $I\in\mathbb{N}$ such that $a_I' \geq C$ and $a_i' < C$ for all $i<I$. Let $C>0$ be sufficiently large. There is a $I\in\mathbb{N}$ with $C\tilde{\gamma}_I>1$ and $a_I' \geq C$ as well as $a_i'<C$ for all $i<I$. Then, by the recurrence relation \eqref{eq:main_recurrence_relation} we obtain
		\begin{eqnarray*}
			\tilde{\lambda}_I(a_{I+1}'-a_I') &=& \tilde{\mu}_I(a_{I}'-a_{I-1}') + \tilde{\gamma}_I a_I' - \xi > \tilde{\gamma}_I C - \xi >0.
		\end{eqnarray*}
		Hence, $a_{I+1}'>a_I'$ and, thus, $\tilde{\gamma}_{I+1} a_{I+1}'> \tilde{\gamma}_I a_I'>1$. Furthermore, there is a $\varepsilon>0$ such that $a_{I+1}'\geq C+\varepsilon$ and $a_{I}'< C+\varepsilon$. Applying the same arguments to $a_{I+1}'$ with a new constant $C'$ set to $C+\varepsilon$ yields inductively that $(a_i')_{i\geq N}$ is strictly increasing, by assumption unbounded and, therefore, divergent to $+\infty$. \par
		Therefore, since two solutions preserve the order of their initial values over time and by the fact that by Proposition \ref{lem:at_most_one_bdd_sol} there is at most one bounded solution to \eqref{eq:main_recurrence_relation}, there is a $\hat{a}>0$ such that for $a_1> \hat{a}$, the solution to \eqref{eq:main_recurrence_relation} tends to $\infty$, while for $a_1< \hat{a}$ it tends to $-\infty$. 
	\end{proof}
	In fact, we can find explicitly the unique initial value $a_1=\hat{a}$ such that the resulting solution to \eqref{eq:main_recurrence_relation} is bounded. We employ generating functions to approach this question. 
	\begin{lemma}\label{thm:explicit_generating_function_form}
		Under Assumption \ref{amp:affine_rates} and with $q = \lambda + \mu + \gamma$ consider the solution $(a_i)_{i\in\mathbb{N}}$ to the recurrence relation \eqref{eq:main_recurrence_relation} with initial value $(0,a)$ for $a\in\mathbb{R}$. Its generating function $\mathcal{E}(z) := \sum_{i=0}^{\infty}a_i z^i$ satisfies
		\begin{equation}
			\mathcal{E}(z) = \dfrac{z\lambda x - z\xi\sum_{k=1}^{\infty}\frac{z^k}{k+\mathfrak{n}}}{\lambda + \mu z^2 - q z}
		\end{equation}
		within its radius of convergence $R\in[0,\infty)$.
	\end{lemma}
	\begin{proof}
		Set $q = \lambda+\gamma+\mu$ and $\tilde{q}_i = q(i+\mathfrak{n})$ such that \eqref{eq:main_recurrence_relation} becomes for $i\geq 1$, 
		\begin{equation*}
			\tilde{q}_i a_{i} = \tilde{\lambda}_i a_{i+1} + \tilde{\mu}_i a_{i-1} + \xi.
		\end{equation*} 
		By exploiting this form of the recurrence relation \eqref{eq:main_recurrence_relation} for $i\geq 1$ we obtain the following.
		\begin{eqnarray*}
			\mathcal{E}(z) &=& \sum_{i=1}^{\infty}a_i z^i  = \sum_{i=1}^{\infty}\tilde{q}_i a_i \dfrac{z^i}{\tilde{q}_i} =  \sum_{i=1}^{\infty}\left(\tilde{\lambda}_i a_{i+1}  + \tilde{\mu}_i a_{i-1} + \xi\right)\dfrac{z^i}{\tilde{q}_i}\\
			&=& \dfrac{\lambda}{q}\sum_{i=1}^{\infty} a_{i+1}z^i  + \dfrac{\mu}{q}\sum_{i=1}^{\infty} a_{i-1}z^i + \xi\sum_{i=1}^{\infty}\dfrac{z^i}{\tilde{q}_i}\\
			&=& \dfrac{\lambda}{qz}\sum_{i=1}^{\infty} a_{i+1}z^{i+1}  + \dfrac{\mu}{q}z\sum_{i=1}^{\infty} a_i z^i + \xi\sum_{i=1}^{\infty}\dfrac{z^i}{\tilde{q}_i}\\
			&=& \dfrac{\lambda}{qz}\left(\mathcal{E}(z)  - xz \right)+ \dfrac{\mu}{q}z\mathcal{E}(z) + \xi\sum_{i=1}^{\infty}\dfrac{z^i}{\tilde{q}_i}.
		\end{eqnarray*}
		Equivalently, we have
		\begin{equation*}
			\mathcal{E}(z) = \dfrac{z\lambda x - z\xi\sum_{k=1}^{\infty}\frac{z^k}{k+\mathfrak{n}}}{\lambda + \mu z^2 - q z}.
		\end{equation*}
	\end{proof}
	Based on the explicit form of the generating function we derive an explicit expression of the coefficients by differentiating in $0$. To justify this approach we have to establish that $\mathcal{E}$ converges within a positive radius of convergence.
	\begin{lemma}\label{minimal_radius_of_convergence_E}
		Under Assumption \ref{amp:affine_rates}	consider the solution $(a_i)_{i\in\mathbb{N}}$ to the recurrence relation \eqref{eq:main_recurrence_relation} with initial value $(0,a)$ for $a>0$. Its generating function $\mathcal{E}$ has a positive radius of convergence $R$.
	\end{lemma}
	\begin{proof}
		Let $\hat{a}$ as in Lemma \ref{prop:behavior_above_below_threshhold}. Consider solutions $(a_i)_{i\in\mathbb{N}}$ and $(a_i')_{i\in\mathbb{N}}$ to \eqref{eq:main_recurrence_relation} with initial conditions $(0,a)$ and $(0,a')$, respectively, where $a>\hat{a}$ as well as $a'<\hat{a}$. Recall $(\Delta_i)_{i\in\mathbb{N}}= (a_i-a_i')_{i\in\mathbb{N}}$. Since there is a $i_0\in\mathbb{N}$ such that $a_i'<0$ for all $i\geq i_0$ it follows for $i\geq i_0$
		\begin{equation*}
			\Delta_i \geq \max\{a_i,|a_i'|\}.
		\end{equation*}
		Additionally $(\Delta_i)_{i\in\mathbb{N}}$ satisfies \eqref{eq:recurrence_xi_zero} so according to Proposition \ref{prop:convergence_speed_xi_zero} it grows like $\Delta_i\sim \underset{\bar{}}{a}^{-i}$ as $i\to\infty$. We obtain that $(a_i)_{i\in\mathbb{N}}$ and $(a_i')_{i\in\mathbb{N}}$ grow at most with speed $\underset{\bar{}}{a}^{-i}$ as $i\to\infty$. Note that the solution $(\hat{a}_i)_{i\in\mathbb{N}}$ of \eqref{eq:main_recurrence_relation} with initial value $(0,\hat{a})$ cannot grow faster than $\underset{\bar{}}{a}^{-i}$ as $i\to\infty$ by positiveness of $\Delta_i$ for all $i\geq 1$. Thus, the series
		\begin{equation*}
			\mathcal{E}(z) = \sum_{i=1}^{\infty}a_i z^i
		\end{equation*}
		converges for any initial value at least within the radius of convergence $R = \underset{\bar{}}{a}$.
	\end{proof}
	Having established a positive radius of convergence we proceed with a direct calculation of the coefficients. Recall that $\underset{\bar{}}{a}<\bar{a}$ are the distinct real zeros of the polynomial $X \mapsto \mu X^2-(\lambda+\mu+\gamma)X+\lambda$.
	\begin{proof}[Proof of Theorem \ref{thm:explicit_form_any_solution}]
		According to \ref{thm:explicit_generating_function_form} and \ref{minimal_radius_of_convergence_E} the introduced generating function $\mathcal{E}$ has a positive radius of convergence and the general form
		\begin{equation*}
			\mathcal{E}(z) = z\dfrac{f(z)}{g(z)}. 
		\end{equation*}
		Note first, that by that form we have 
		\begin{equation*}
			\partial_z^i\mathcal{E}(z)\bigg\vert_{z=0} = i\,\partial_z^{i-1}\left(\dfrac{f(z)}{g(z)}\right)\bigg\vert_{z=0}= i\sum_{k=0}^{i-1}\binom{i-1}{k}\left(\partial_z^{k}f(z)\partial_z^{i-1-k} (g(z)^{-1})\right)\bigg\vert_{z=0} .
		\end{equation*}
		Set $ c  = (\sqrt{(\lambda+\mu+\gamma)^2-4\lambda\mu})^{-1}$. Then we have for $g(z)=\lambda-q z + \mu z^2$ the expression
		\begin{equation*}
			g(z)^{-1} =\dfrac{ c }{z-\bar{a}} - \dfrac{ c }{z-\underset{\bar{}}{a}},
		\end{equation*}
		which gives canonically any derivative of order $n$ of $g(z)^{-1}$ at $0$, namely,
		\begin{equation}
			\partial_z^i g(z)^{-1}\bigg\vert_{z=0} =  c  i!\left(-\dfrac{1}{\bar{a}^{i+1}}+ \dfrac{1}{\underset{\bar{}}{a}^{i+1}}\right).
		\end{equation}
		Considering the numerator $f(z)= \lambda x - \xi\sum_{k=1}^{\infty}\frac{z^k}{k+\mathfrak{n}}$, we find that for $i\geq 1$ we have
		\begin{eqnarray*}
			\partial_z^i f(z)\bigg\vert_{z=0} = -\xi\sum_{k=i}^{\infty} \dfrac{z^{k-i}}{k+ \mathfrak{n} }k\cdot(k-1)\cdot\hdots\cdot (k-i+1)\bigg\vert_{z=0} = -\xi\dfrac{i!}{i+ \mathfrak{n} }.
		\end{eqnarray*}
		Hence,
		\begin{eqnarray*}
			\partial_z^i\mathcal{E}(z)\bigg\vert_{z=0} &=& i\,\partial_z^{i-1}\left(\dfrac{f(z)}{g(z)}\right)\bigg\vert_{z=0}\\
			&=& i\left(\lambda\,x\,(i-1)! c \left(\dfrac{1}{\underset{\bar{}}{a}^i}-\dfrac{1}{\bar{a}^i}\right)\right. \nonumber\\
			&&\qquad \left. -\sum_{k=1}^{i-1}\binom{i-1}{k} \dfrac{k!\,\xi}{k+ \mathfrak{n} } c  (i-1-k)!\left(-\dfrac{1}{\bar{a}^{i-k}}+ \dfrac{1}{\underset{\bar{}}{a}^{i-k}}\right)\right)\nonumber\\
			&=& i!\left(\lambda\,x\, c \left(\dfrac{1}{\underset{\bar{}}{a}^i}-\dfrac{1}{\bar{a}^i}\right)-\sum_{k=1}^{i-1}\dfrac{ c\,\xi }{k+ \mathfrak{n} } \left(-\dfrac{1}{\bar{a}^{i-k}}+ \dfrac{1}{\underset{\bar{}}{a}^{i-k}}\right)\right)\nonumber\\
		\end{eqnarray*}
		Since $a_i= i!^{-1}\partial_z^i\mathcal{E}(z)\bigg\vert_{z=0} $, we find that for any $i\geq 1$
		\begin{equation*}
			a_i = \lambda\,x\,c \left(\dfrac{1}{\underset{\bar{}}{a}^i}-\dfrac{1}{\bar{a}^i}\right)-\sum_{k=1}^{i-1}\dfrac{ c\,\xi }{k+ \mathfrak{n} } \left(-\dfrac{1}{\bar{a}^{i-k}}+ \dfrac{1}{\underset{\bar{}}{a}^{i-k}}\right).
		\end{equation*}
	\end{proof}
	To improve readability, we use the following notation. For $z\in\mathbb{R}$ with $|z|<1$ and $n\in\mathbb{N}, n\geq 1$ we use the notation
	\begin{equation}
		\Phi_{n}(z) := \sum_{k=0}^{\infty}\dfrac{z^k}{n+k},
	\end{equation}
	leaning on the  Lerch transcendent $\Phi(n,1,z)$, see \cite{Lerch1900}. 
	\begin{lemma}\label{lem:small_initial_value}
		Under Assumption \ref{amp:affine_rates} any sequence $(a_i)_{i\in\mathbb{N}}$ satisfying \eqref{eq:main_recurrence_relation} with initial values $x_0 = 0$ and $a_1<\hat{a}$ where
		\begin{equation}\label{eq:x_hat_minimal_solution_Phi}
			\hat{a}	:= \dfrac{\xi}{\lambda}\left(\Phi_{\mathfrak{n}}(\underset{\bar{}}{a})-  \mathfrak{n} ^{-1}\right).
		\end{equation}	 	
		tends to $-\infty$ as $i\to\infty$.
	\end{lemma}
	\begin{proof}
		Assume that $a_1 < \dfrac{1}{\lambda}\left(\Phi_{\mathfrak{n}}(\underset{\bar{}}{a}) -  \mathfrak{n} ^{-1}\right)$, i.e., there is a $\delta > 0$ such that
		\begin{equation}\label{initial_value_relation_delta}
			a_1 = \dfrac{\xi}{\lambda}\left(\Phi_{\mathfrak{n}}(\underset{\bar{}}{a}) -  \mathfrak{n} ^{-1}\right)-\delta.
		\end{equation}
		Since $\Phi_{\mathfrak{n}}(\cdot)$ is continuous and $\underset{\bar{}}{a} < 1$, we find that its left limit for $z\nearrow \underset{\bar{}}{a}$ exists such that
		\begin{equation*}
			\forall \varepsilon > 0\,\exists\, \bar{\delta}> 0\,\forall z:\; \underset{\bar{}}{a} - z < \bar{\delta} \Rightarrow  \Phi_{\mathfrak{n}}(\underset{\bar{}}{a})  - \Phi_{\mathfrak{n}}(z)  <  \varepsilon. 
		\end{equation*}
		Note that $\mu z^2 - q z + \lambda > 0$ for all $z<\underset{\bar{}}{a}$. Choose $\varepsilon = \dfrac{\delta\lambda}{2\xi}$ such that for all $z\in [0,\underset{\bar{}}{a})\cap (\underset{\bar{}}{a}-\bar{\delta},\underset{\bar{}}{a})$ we obtain by \eqref{initial_value_relation_delta} that
		\begin{eqnarray*}
			\mathcal{E}(z) &=&  \dfrac{z\left(\xi\Phi_{\mathfrak{n}}(\underset{\bar{}}{a})  -  \xi\mathfrak{n} ^{-1}-\lambda\delta - \xi\Phi_{\mathfrak{n}}(z)  +  \xi\mathfrak{n} ^{-1}\right)}{\lambda + \mu z^2 - q z}\\
			&<&  \dfrac{z\left(\xi\varepsilon-\lambda\delta \right)}{\lambda + \mu z^2 - q z} = -\dfrac{z\lambda\delta}{2(\lambda + \mu z^2 - q z)} < 0.
		\end{eqnarray*}
		Since positivity of the sequence $(a_i)_{i\in\mathbb{N}}$ would imply that $\mathcal{E}(z)>0$ for all $z>0$ we find that for $a_1 < \dfrac{1}{\lambda}\left(\Phi_{\mathfrak{n}}(\underset{\bar{}}{a}) -  \mathfrak{n} ^{-1}\right)$ there is an index $i_0\in\mathbb{N}$ such that $x_{i_0} < 0$. Consequently, $a_i$ tends to $-\infty$ as $i\to\infty$ by Lemma \ref{prop:behavior_above_below_threshhold}.
	\end{proof}
	The necessity of $\xi>0$ becomes evident in the choice of $\varepsilon$ in the proof. In fact, it turns out that the value $\hat{a}$ yields the minimal positive solution of \eqref{eq:main_recurrence_relation} because any other smaller initial value leads to a sequence diverging to $-\infty$. 
	\begin{proof}[Proof of Theorem \ref{thm:minimal_solution_result}]
		Take $(a_0,a_1)=(0,\hat{a})$. In Theorem \ref{thm:explicit_form_any_solution} the following identity is given.
		\begin{eqnarray}
			a_i &=& \lambda\, \hat{a}\, c \left(\dfrac{1}{\underset{\bar{}}{a}^i}-\dfrac{1}{\bar{a}^i}\right)-\sum_{k=1}^{i-1}\dfrac{ c\xi }{k+ \mathfrak{n} } \left(-\dfrac{1}{\bar{a}^{i-k}}+ \dfrac{1}{\underset{\bar{}}{a}^{i-k}}\right)\nonumber\\
			&=& \label{eq:extended_form_of_ai}\sum_{k=1}^{\infty} \dfrac{\underset{\bar{}}{a}^{k}}{k+\mathfrak{n}}\, c\xi \left(\dfrac{1}{\underset{\bar{}}{a}^i}-\dfrac{1}{\bar{a}^i}\right)-\sum_{k=1}^{i-1}\dfrac{ c\xi }{k+ \mathfrak{n} } \left(-\dfrac{1}{\bar{a}^{i-k}}+ \dfrac{1}{\underset{\bar{}}{a}^{i-k}}\right).
		\end{eqnarray}
		Define for $i\geq 1$ the sequence $(b_i)_{i\geq 1}$ with $b_i = \sum_{k=1}^{i-1}\dfrac{ \bar{a}^{k} }{k+ \mathfrak{n} }$. Then
		\begin{eqnarray*}
			\sum_{k=1}^{i-1}\dfrac{ 1 }{k+ \mathfrak{n} } \dfrac{1}{\bar{a}^{i-k}} &=& \dfrac{1}{\bar{a}^{i}}\sum_{k=1}^{i-1}\dfrac{ \bar{a}^{k} }{k+ \mathfrak{n} } = \dfrac{b_i}{\bar{a}^{i}}.
		\end{eqnarray*}
		Furthermore, it holds
		\begin{eqnarray}
			\dfrac{b_{i+1}-b_i}{\bar{a}^{i+1}-\bar{a}^{i}} &=& \dfrac{1}{\bar{a}^{i}(\bar{a}-1)}\dfrac{\bar{a}^{i}}{i+\mathfrak{n}} = \dfrac{1}{(\bar{a}-1)(i+\mathfrak{n})}\to 0,\qquad i\to\infty.
		\end{eqnarray}
		Using $\bar{a}>1$ and employing the Stolz-C\'esaro theorem we find that
		\begin{equation*}
			\dfrac{1}{\bar{a}^{i}}\sum_{k=1}^{i-1}\dfrac{ \bar{a}^{k} }{k+ \mathfrak{n} } \to 0,\;\text{as } i\to\infty,
		\end{equation*}
		as well. Thus, the limiting behavior of $(a_i)_{i\geq 0}$ is the same as of
		\begin{equation}\label{eq:lower_bound_extended_ai}
			\dfrac{c}{\underset{\bar{}}{a}^i}\left(\lambda\, \hat{a} - \sum_{k=1}^{\infty} \dfrac{\xi\underset{\bar{}}{a}^{k}}{k+\mathfrak{n}}\right)= \dfrac{c}{\underset{\bar{}}{a}^i}\sum_{k=i}^{\infty} \dfrac{\xi\underset{\bar{}}{a}^{k}}{k+\mathfrak{n}}\geq \dfrac{c\xi}{i+\mathfrak{n}}.
		\end{equation}	 
		Secondly, using the integral comparison, it holds 
		\begin{eqnarray}
			\sum_{k=i}^{\infty} \dfrac{\underset{\bar{}}{a}^{k}}{k+\mathfrak{n}} &\leq & \dfrac{\underset{\bar{}}{a}^{i}}{i+\mathfrak{n}} + \int_i^{\infty}\dfrac{\underset{\bar{}}{a}^{s}}{s+\mathfrak{n}} ds\nonumber\\
			&\leq & \dfrac{\underset{\bar{}}{a}^{i}}{i+\mathfrak{n}} + \dfrac{\underset{\bar{}}{a}^{-\mathfrak{n}}}{i+\mathfrak{n}}\int_{i+\mathfrak{n}}^{\infty}\mathrm{exp}(\log\underset{\bar{}}{a}{s}) ds\nonumber\\
			&= &\label{eq:link_to_log} \dfrac{\underset{\bar{}}{a}^{i}}{i+\mathfrak{n}} + \dfrac{\underset{\bar{}}{a}^{i}}{(i+\mathfrak{n})(-\log\underset{\bar{}}{a})}.
		\end{eqnarray}
		Hence, $\dfrac{c\xi}{\underset{\bar{}}{a}^i}\sum_{k=i}^{\infty} \dfrac{\underset{\bar{}}{a}^{k}}{k+\mathfrak{n}}\to 0$ as $i\to\infty$ such that $a_i\to 0$ as $i\to\infty$ and moreover $a_i >0$ for all $i\in\mathbb{N}$, since otherwise $a_i\to-\infty$ as $i\to\infty$ by Lemma \ref{lem:small_initial_value}. Moreover, by the lower bound from \eqref{eq:lower_bound_extended_ai} and the upper bound of order $\mathcal{O}(i^{-1})$ it follows that $a_i = \mathcal{O}(i^{-1})$ as $i\to\infty$. Finally by Lemma \ref{prop:behavior_above_below_threshhold}, every solution starting from a smaller $a_1$ tends to $-\infty$, such that $(a_i)_{i\geq 0}$ with initial value $(0,\hat{a})$ is the minimal positive solution to \eqref{eq:main_recurrence_relation}.
	\end{proof}
\section{Issues with simulating the minimal positive solution}
	This section is dedicated to the proof of Theorem \ref{thm:exclusion_rational_of_both} and to the stability analysis of said recurrence relation with respect to perturbations of the value $\hat{a}$ and its implication on simulation results. 
	To make our point, we recall a preliminary lemma and an important theorem on the nature of algebraic and transcendent numbers. They can be found in \cite{Ono1990} together with the respective proofs.
	\begin{lemma}[Properties Algebraic \& Transcendent Numbers, \cite{Ono1990}]\label{lem:structure_algebraic_numbers}
		The algebraic and transcendental numbers satisfy the 
		\begin{itemize}
			\item[a)] Any $p\in\mathbb{Q}$ is algebraic.
			\item[b)] The sum of a transcendental number and an algebraic number is transcendental.
			\item[c)] The product of a transcendental number by a nonzero algebraic number is transcendental.
		\end{itemize}
	\end{lemma}
	Additionally, we employ a theorem on the relation of algebraic and transcendental numbers via the logarithm. It can be found including the proof in \cite{Kuhne2015}.
	\begin{theorem}[Logarithm of Algebraic Numbers, \cite{Kuhne2015}]\label{thm:logarithm_algebraic_transcendent}
		Let $z\neq 1$ be an algebraic complex number. Then $\log(z)$ is transcendental.
	\end{theorem}
	The combination of both, together with the explicit form of solutions, which we obtained previously, leads to the following proof of Theorem \ref{thm:exclusion_rational_of_both}.
	\begin{proof}
		Assume $\hat{a}\in\mathbb{Q}$ and $\gamma,\lambda,\mu\in\mathbb{Q}$. Thus, $\underset{\bar{}}{a}$ is algebraic and, therefore, also $1-\underset{\bar{}}{a}$ is algebraic by Lemma \ref{lem:structure_algebraic_numbers}. We have
		\begin{eqnarray*}
			\sum_{i=1}^{\infty} \dfrac{(\underset{\bar{}}{a})^i}{i+ \mathfrak{n} } &=&  \dfrac{1}{\underset{\bar{}}{a}}\sum_{i=1}^{\infty} \dfrac{(\underset{\bar{}}{a})^{i+1}}{i+ \mathfrak{n} } = \dfrac{1}{\underset{\bar{}}{a}}\left(\sum_{i=1}^{\infty} \dfrac{(\underset{\bar{}}{a})^{i}}{i-1+ \mathfrak{n} }-\dfrac{\underset{\bar{}}{a}}{ \mathfrak{n} }\right)\\
			&=& \dfrac{1}{\underset{\bar{}}{a}}\left(\dfrac{1}{\underset{\bar{}}{a}}\left(\sum_{i=1}^{\infty} \dfrac{(\underset{\bar{}}{a})^{i}}{i-2+ \mathfrak{n} }-\dfrac{\underset{\bar{}}{a}}{ \mathfrak{n} -1}\right)-\dfrac{\underset{\bar{}}{a}}{ \mathfrak{n} }\right)\\
			&=& \dfrac{1}{\underset{\bar{}}{a}}\left(\dfrac{1}{\underset{\bar{}}{a}}\left(\hdots\left(\sum_{i=1}^{\infty} \dfrac{(\underset{\bar{}}{a})^{i}}{i+1}-\dfrac{\underset{\bar{}}{a}}{2}\right)\hdots\right)-\dfrac{\underset{\bar{}}{a}}{ \mathfrak{n} }\right).
		\end{eqnarray*}
		Consider the term
		\begin{eqnarray*}
			\sum_{i=1}^{\infty} \dfrac{(\underset{\bar{}}{a})^{i}}{i+1} &=& \dfrac{1}{\underset{\bar{}}{a}}\sum_{i=1}^{\infty} \dfrac{(\underset{\bar{}}{a})^{i+1}}{i+1}=\dfrac{1}{\underset{\bar{}}{a}}\int_0^{\underset{\bar{}}{a}} \dfrac{1}{1-x}-1 dx\\
			&=& \dfrac{1}{\underset{\bar{}}{a}}(-\ln(1-\underset{\bar{}}{a}))-1.
		\end{eqnarray*}
		By Theorem \ref{thm:logarithm_algebraic_transcendent}, since $1-\underset{\bar{}}{a}\neq 1$, we obtain that $\sum_{i=1}^{\infty} \dfrac{(\underset{\bar{}}{a})^{i}}{i+1}$ is transcendent. Due to $\underset{\bar{}}{a}$ being algebraic, the term  $\sum_{i=1}^{\infty} \dfrac{(\underset{\bar{}}{a})^{i}}{i+\mathfrak{n}}$ is transcendent by the previous reasoning and lemma \ref{lem:structure_algebraic_numbers}. Thus,
		\begin{equation*}
			\lambda\,\hat{a} = \sum_{i=1}^{\infty} \dfrac{(\underset{\bar{}}{a})^{i}}{i+\mathfrak{n}}
		\end{equation*}
		is transcendent, but by assumption $\lambda \hat{a}\in\mathbb{Q}$ and, therefore, algebraic, such that we arrive at a contradiction.
	\end{proof} 
	Figure \ref{fig:divergence_from_initial_value} shows the simulation of a few trajectories based on the recurrence relation \eqref{eq:main_recurrence_relation} with varying values of $a>0$. Therein, we choose $a$ close to $\hat{a}$, for which the sequence stays bounded in view of Theorem \ref{thm:minimal_solution_result}. We show in Theorem \ref{thm:exclusion_rational_of_both} that a simulation with $a_1=\hat{a}$ is not feasible. 
	\begin{figure}[ht]
		\centering
\begin{tikzpicture}

\definecolor{color0}{rgb}{0.12156862745098,0.466666666666667,0.705882352941177}
\definecolor{color1}{rgb}{1,0.498039215686275,0.0549019607843137}
\definecolor{color2}{rgb}{0.172549019607843,0.627450980392157,0.172549019607843}
\definecolor{color3}{rgb}{0.83921568627451,0.152941176470588,0.156862745098039}
\definecolor{color4}{rgb}{0.580392156862745,0.403921568627451,0.741176470588235}
\definecolor{color5}{rgb}{0.549019607843137,0.337254901960784,0.294117647058824}
\definecolor{color6}{rgb}{0.890196078431372,0.466666666666667,0.76078431372549}
\definecolor{color7}{rgb}{0.737254901960784,0.741176470588235,0.133333333333333}
\definecolor{color8}{rgb}{0.0901960784313725,0.745098039215686,0.811764705882353}
\definecolor{color9}{rgb}{0.96078431372549,0.870588235294118,0.701960784313725}

\begin{axis}[
scale only axis,
scale = 1,
legend cell align={left},
legend entries={{a = 0.914789},{a = 0.917011},{a = 0.919233},{a = 0.921456},{a = 0.923678},{a = 0.925900},{a = 0.928122},{a = 0.930344},{a = 0.932567},{a = 0.934789}},
legend style={at={(0.03,0.97)}, anchor=north west, draw=white!80.0!black},
tick align=outside,
tick pos=left,
x grid style={white!69.019607843137251!black},
xlabel={Iteration Steps},
xmin=-0.45, xmax=9.45,
y grid style={white!69.019607843137251!black},
ylabel={Value of Sequence},
ymin=-1793.14661371236, ymax=1794.09254940317,
/pgf/number format/.cd,
1000 sep={}
]
\addlegendimage{no markers, color0}
\addlegendimage{no markers, color1}
\addlegendimage{no markers, color2}
\addlegendimage{no markers, color3}
\addlegendimage{no markers, color4}
\addlegendimage{no markers, color5}
\addlegendimage{no markers, color6}
\addlegendimage{no markers, white!49.803921568627452!black}
\addlegendimage{no markers, color7}
\addlegendimage{no markers, color8}
\addplot [semithick, color0]
table [row sep=\\]{%
0	0 \\
1	0.914788841098104 \\
2	0.874695335717945 \\
3	0.617582382500826 \\
4	-0.183938241926967 \\
5	-3.43888756458677 \\
6	-17.7085198519429 \\
7	-81.2140405983673 \\
8	-364.62614150799 \\
9	-1630.0902881162 \\
};
\addplot [semithick, color1]
table [row sep=\\]{%
0	0 \\
1	0.917011063320326 \\
2	0.884917557940167 \\
3	0.663271271389717 \\
4	0.0200973136285952 \\
5	-2.52773734236451 \\
6	-13.6396501630539 \\
7	-63.0439301628111 \\
8	-283.484955317766 \\
9	-1267.7428979025 \\
};
\addplot [semithick, color2]
table [row sep=\\]{%
0	0 \\
1	0.919233285542548 \\
2	0.895139780162389 \\
3	0.708960160278604 \\
4	0.224132869184144 \\
5	-1.61658712014232 \\
6	-9.57078047416517 \\
7	-44.8738197272562 \\
8	-202.343769127546 \\
9	-905.395507688819 \\
};
\addplot [semithick, color3]
table [row sep=\\]{%
0	0 \\
1	0.921455507764771 \\
2	0.905362002384612 \\
3	0.754649049167495 \\
4	0.428168424739708 \\
5	-0.705436897920064 \\
6	-5.50191078527612 \\
7	-26.7037092916999 \\
8	-121.202582937321 \\
9	-543.048117475117 \\
};
\addplot [semithick, color4]
table [row sep=\\]{%
0	0 \\
1	0.923677729986993 \\
2	0.915584224606834 \\
3	0.800337938056381 \\
4	0.632203980295252 \\
5	0.20571332430211 \\
6	-1.43304109638744 \\
7	-8.53359885614532 \\
8	-40.0613967471027 \\
9	-180.700727261447 \\
};
\addplot [semithick, color5]
table [row sep=\\]{%
0	0 \\
1	0.925899952209215 \\
2	0.925806446829055 \\
3	0.846026826945269 \\
4	0.836239535850802 \\
5	1.1168635465243 \\
6	2.63582859250132 \\
7	9.63651157940967 \\
8	41.079789443117 \\
9	181.646662952231 \\
};
\addplot [semithick, color6]
table [row sep=\\]{%
0	0 \\
1	0.928122174431437 \\
2	0.936028669051278 \\
3	0.891715715834159 \\
4	1.04027509140636 \\
5	2.02801376874656 \\
6	6.70469828139035 \\
7	27.8066220149659 \\
8	122.220975633342 \\
9	543.994053165933 \\
};
\addplot [semithick, white!49.803921568627452!black]
table [row sep=\\]{%
0	0 \\
1	0.93034439665366 \\
2	0.946250891273501 \\
3	0.93740460472305 \\
4	1.24431064696193 \\
5	2.93916399096881 \\
6	10.7735679702794 \\
7	45.976732450522 \\
8	203.362161823567 \\
9	906.341443379633 \\
};
\addplot [semithick, color7]
table [row sep=\\]{%
0	0 \\
1	0.932566618875882 \\
2	0.956473113495722 \\
3	0.983093493611936 \\
4	1.44834620251747 \\
5	3.85031421319099 \\
6	14.8424376591681 \\
7	64.1468428860767 \\
8	284.503348013785 \\
9	1268.68883359331 \\
};
\addplot [semithick, color8]
table [row sep=\\]{%
0	0 \\
1	0.934788841098104 \\
2	0.966695335717945 \\
3	1.02878238250083 \\
4	1.65238175807304 \\
5	4.76146443541325 \\
6	18.9113073480571 \\
7	82.316953321633 \\
8	365.644534204011 \\
9	1631.03622380701 \\
};
\node at (400,50)[
  text width=85pt,
  draw=black,
  line width=0.4pt,
  inner sep=3pt,
  anchor=base west,
  text=black,
  rotate=0.0,
  align=left
]{ $\mathfrak{n}=5.00$
$\mu=0.03$
$\gamma=0.15$
$\lambda=0.05$};
\end{axis}

\end{tikzpicture}
		\caption{Simulations of solutions $(a_i)_{i\in\mathbb{N}}$ to the recurrence relation \eqref{eq:main_recurrence_relation} as functions of $i$ using varying values of $a$. We observe that the trajectories exhibit one of two distinct behaviors, either a tendency to grow or decay increasingly fast, attaining values as large as $1500$ or as small as $-1500$ after just a few simulation steps. 
		\label{fig:divergence_from_initial_value}}
	\end{figure}
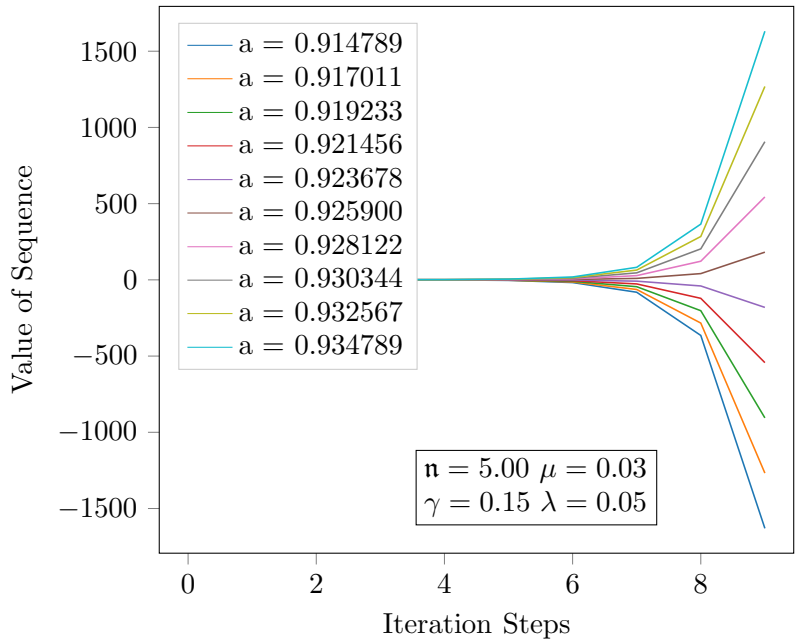
	The key problem is that a computer only has finite accuracy. From Theorem \ref{thm:minimal_solution_result} follows, that there is exactly one initial value $a$, such that the solution of the recurrence relation \eqref{eq:main_recurrence_relation} stays bounded. Thus, to simulate the solution $\hat{a}$ has to be interpreted exactly by the computer, i.e., it should at least be rational.\\
	In the case of $\gamma$, $\lambda$ or $\mu$ being irrational, we cannot initialize the simulation exactly. The same is true, if $\hat{a}$ is irrational. Thus, in order to attempt a simulation without an error we need at least $\hat{a}\in\mathbb{Q}$ and $\gamma,\lambda,\mu\in\mathbb{Q}$ to simulate the minimal positive solution $(a_i)_{i\in\mathbb{N}}$.\\
	Theorem \ref{thm:exclusion_rational_of_both} underlines the observation made in Figure \ref{fig:divergence_from_initial_value} that it is not possible to simulate the minimal positive solution correctly. This concludes our analysis of the recurrence relation \eqref{eq:main_recurrence_relation} and the minimal positive solution when fixing $a_0=0$ under Assumption \ref{amp:affine_rates}.
\section{Outlook}
	Having found explicit results under Assumption \ref{amp:affine_rates} nautral extensions come to mind by generalizing the form of the parameters. Indeed, they are twofold. Firstly, considering the most general form of the non-homogenity $\xi$ as a function of $i\in\mathbb{N}$, which could be called $\xi_i$, give identical qualitative results on the form of the minimal solution by changing the terms $\xi\Phi_{\mathfrak{n}}(z)$ to 
	\begin{equation*}
		\Psi_{\mathfrak{n};(\xi_i)_{i\in\mathbb{N}}}(z)=\sum_{i=0}^{\infty} \dfrac{\xi_i z^i}{i+\mathfrak{n}}.
	\end{equation*}
	This yields identical results as in Theorems \ref{thm:explicit_form_any_solution} and \ref{thm:minimal_solution_result} for the form of the resulting sequence $(a_i)_{i\in\mathbb{N}}$. Nonetheless, the quantitative result on the convergence speed cannot be proven in the same way and will depend on the behavior of $(\xi_i)_{i\in\mathbb{N}}$. In particular, the link to the natural logarithm used in Equation \eqref{eq:link_to_log} is no longer valid. Methods from Analytic Combinatorics as discussed, for example, in \cite{flajolet2009analytic} might shed light on the question of the link between the limit behaviors of $(a_i)_{i\in\mathbb{N}}$ and $(\xi_i)_{i\in\mathbb{N}}$ when keeping the remaining assumptions in Assumption \ref{amp:affine_rates}. \\
	On the other hand, using more general forms for $(\tilde{\lambda}_i)_{i\in\mathbb{N}}, (\tilde{\mu}_i)_{i\in\mathbb{N}}$ and $(\tilde{\gamma}_i)_{i\in\mathbb{N}}$ will render the calculations more difficult or even unfeasible, as soon as their quotients are not independent of $i$. The particular affine form, which we chose in Assumption \ref{amp:affine_rates}, is only a first step and a useful when looking for explicit forms of the solution.   
	
\medskip

\bibliographystyle{alpha}
\bibliography{bibliography}

\end{document}